\documentclass[10pt,reqno]{amsart}

\usepackage{amsmath,amsthm,amssymb,comment,a4wide}
\usepackage{verbatim}
\usepackage{braket}
\usepackage{mathtools}

\usepackage{caption}
\usepackage{times}
\usepackage[T1]{fontenc}
\usepackage{mathrsfs}
\usepackage{latexsym}
\usepackage[dvips]{graphics}
\usepackage{epsfig}
\usepackage{amsmath,amsfonts,amsthm,amssymb,amscd}
\input amssym.def
\input amssym.tex
\usepackage{color}
\usepackage{hyperref}
\usepackage{url}
\newcommand{\bburl}[1]{\textcolor{blue}{\url{#1}}}

\usepackage{tikz}
\usepackage{tkz-tab}
\usepackage{tkz-graph}
\usetikzlibrary{shapes.geometric,positioning}

\newcommand{\burl}[1]{\textcolor{blue}{\url{#1}}}

\numberwithin{equation}{section}

\newtheorem{thm}{Theorem}[section]

\theoremstyle{plain}

\newtheorem{proposition}[thm]{Proposition}
\newtheorem{theorem}[thm]{Theorem}

\newcommand\be{\begin{equation}}
\newcommand\ee{\end{equation}}
\newcommand\bee{\begin{equation*}}
\newcommand\eee{\end{equation*}}
\newcommand\bea{\begin{eqnarray}}
\newcommand\eea{\end{eqnarray}}
\newcommand\beae{\begin{eqnarray*}}
\newcommand\eeae{\end{eqnarray*}}
\newcommand\bi{\begin{itemize}}
\newcommand\ei{\end{itemize}}
\newcommand\ben{\begin{enumerate}}
\newcommand\een{\end{enumerate}}
\newcommand\bc{\begin{center}}
\newcommand\ec{\end{center}}
\newcommand\ba{\begin{array}}
\newcommand\ea{\end{array}}

\newcommand\frakfamily{\usefont{U}{yfrak}{m}{n}}
\DeclareTextFontCommand{\textfrak}{\frakfamily}

\newcommand{\hr}[1]{\href{#1}{\url{#1}}}

\title[]{On Algorithms to Calculate Integer Complexity}

\author[]{Katherine Cordwell}
\email{\textcolor{blue}{\href{mailto:kcordwel@cs.cmu.edu}{kcordwel@cs.cmu.edu}}}
\address{Department of Mathematics, University of Maryland, College Park, MD 20742}

\author[]{Alyssa Epstein}
\email{\textcolor{blue}{\href{mailto:alye@stanford.edu}{alye@stanford.edu}}}
\address{Department of Mathematics and Statistics, Williams College, Williamstown, MA 01267}

\author[]{Anand Hemmady}
\email{\textcolor{blue}{\href{mailto:ash6@williams.edu}{ash6@williams.edu}}}
\address{Department of Mathematics and Statistics, Williams College, Williamstown, MA 01267}

\author[]{Steven J. Miller}
\email{\textcolor{blue}{\href{mailto:sjm1@williams.edu}{sjm1@williams.edu}},  \textcolor{blue}{\href{Steven.Miller.MC.96@aya.yale.edu}{Steven.Miller.MC.96@aya.yale.edu}}}
\address{Department of Mathematics and Statistics, Williams College, Williamstown, MA 01267}

\author[]{Eyvindur Palsson}
\email{\textcolor{blue}{\href{mailto:palsson@vt.edu}{palsson@vt.edu}}}
\address{Department of Mathematics, Virginia Polytechnic Institute and State University, VA 24061}

\author[]{Aaditya Sharma}
\email{\textcolor{blue}{\href{mailto:as2718@cam.ac.uk}{as2718@cam.ac.uk}}}
\address{Department of Mathematics and Statistics, Williams College, Williamstown, MA 01267}

\author[]{Stefan Steinerberger}
\email{\textcolor{blue}{\href{mailto:stefan.steinerberger@yale.edu}{stefan.steinerberger@yale.edu}}}
\address{Department of Mathematics, Yale University, CT 06510}

\author[]{Yen Nhi Truong Vu}
\email{\textcolor{blue}{\href{mailto:ytruongvu17@amherst.edu}{ytruongvu17@amherst.edu}}}
\address{Department of Mathematics, Amherst College, Amherst, MA 01002}

\thanks{}

\subjclass[2010]{60B10, 11B39, 11B05  (primary) 65Q30 (secondary)}

\keywords{Integer Complexity}

\date{\today}

\begin{document}

\begin{abstract} 
 We consider a problem first proposed by Mahler and Popken in 1953 and later developed by Coppersmith, Erd\H{o}s, Guy, Isbell, Selfridge, and others.  Let $f(n)$ be the complexity of $n \in \mathbb{Z^{+}}$, where $f(n)$ is defined as the least number of $1$'s needed to represent $n$ in conjunction with an arbitrary number of $+$'s, $*$'s, and parentheses.  Several algorithms have been developed to calculate the complexity of all integers up to $n$.  Currently, the fastest known algorithm runs in time $\mathcal{O}(n^{1.230175})$ and was given by J. Arias de Reyna and J. van de Lune in 2014.  This algorithm makes use of a recursive definition given by Guy and iterates through products, $f(d) + f\left(\frac{n}{d}\right)$, for $d \ |\ n$, and sums, $f(a) + f(n - a)$, for $a$ up to some function of $n$.  The rate-limiting factor is iterating through the sums.  We discuss potential improvements to this algorithm via a method that provides a strong uniform bound on the number of summands that must be calculated for almost all $n$.  We also develop code to run J. Arias de Reyna and J. van de Lune's analysis in higher bases and thus reduce their runtime of $\mathcal{O}(n^{1.230175})$ to $\mathcal{O}(n^{1.222911236})$.  All of our code can be found online at: https://github.com/kcordwel/Integer-Complexity.
\end{abstract}

\maketitle



\section{Introduction}
\subsection{Introduction}
In this paper, $\log$ denotes $\ln$, and $\log_b$ denotes the logarithm in base b.  Given $n \in \mathbb{N}$, the complexity of $n$, which we denote $f(n)$, is defined as the least number of 1's needed to represent $n$ using an arbitrary number of additions, multiplications, and parentheses.  For example, because 6 may be represented as $(1 + 1)(1 + 1 + 1)$, $f(6) \leq 5$.
Calculating $f(n)$ for arbitrary $n$ is a problem that was posed in 1953 by Mahler and Popken \cite{MP}.  Guy \cite{G} drew attention to this problem in 1986 when he discussed it and several other simply stated problems in an Am. Math. Monthly article.  The following recursive expression for integer complexity highlights the interplay of additive and multiplicative structures:

\begin{equation}
 f(n) = \min\limits_{\substack{d\ |\ n \\ 2 \leq d \leq \sqrt{n} \\ 1 \leq a \leq n/2}} \left\lbrace f(d) + f\left(\frac{n}{d}\right), f(a) + f(n - a)\right\rbrace.
\end{equation}

Some unconditional bounds on $f(n)$ are known.  In particular, \cite{G} attributes a lower bound of $f(n) \geq 3\log_3(n)$ to Selfridge.  
Also, an upper bound of $f(n) \leq 3\log_2(n)$ is attributed to Coppersmith.  Extensive numerical investigation (see \cite{IBCOOP}) suggests that $ f(n) \sim 3.3 \log_{3}(n)$ for $n$ large
but it is not even known whether $f(n) \geq (3+\varepsilon_0) \log_{3}{n}$ for some $\varepsilon_0 > 0$. As a step towards understanding these problems Altman and Zelinsky \cite{AZ} introduced the discrepancy $\delta(n) = f(n) - 3\log_3(n)$ and provided a way to classify those numbers with a small discrepancy. This classification was taken further by Altman \cite{A1,A2} where he obtained a finite set of polynomials that represent precisely the numbers with small defects. As a consequence Altman \cite{A3} was able to calculate the integer complexity of certain classes of numbers. Any progress on these difficult questions likely requires a substantial new idea; the main difficulty,
the interplay between additive and multiplicative structures, is at the core of a variety of different open problems, which we believe adds to its allure.

\subsection{Algorithms.} Much of the progress on this problem has been algorithmic.  Using the above recursive definition, it is possible to write algorithms to calculate $f(n)$ for large values of $n$ where the rate-limiting step of the algorithm is iterating through the summands, $f(a) + f(n-a)$, for many values of $a$.  In particular, the brute-force algorithm that iterates over all $a's$ such that $1 \leq a \leq n/2$ runs in time $\mathcal{O}{(n^2)}$, but there are ways to bound the number of summands that must be checked so as to significantly decrease the computational complexity.  Srinivar and Shankar \cite{SS} used the unconditional upper and lower bounds on $f(n)$ to bound the number of summands, obtaining an algorithm that runs in time $\mathcal{O}(n^{\log_2(3)}) < \mathcal{O}(n^{1.59})$.\\

The fastest known algorithm runs in time $\mathcal{O}(n^{1.230175})$ and is due to J. Arias de Reyna and J. van de Lune \cite{AV}.  Also, the experimental data in \cite{IBCOOP} is based on an algorithm that calculates $f(n)$ for $n$ up to $n = 10^{12}$.  They derive many interesting results from their data, but they do not analyze the runtime of their algorithm.
We obtain both an overall improvement on the runtime of the J. Arias de Reyna and J. van de Lune algorithm and a potential internal improvement to the workings of the algorithm.  The overall improvement is derived from running the analysis of \cite{AV} in much higher bases, while the internal improvement gives a strong uniform bound on the number of summands $f(a) + f(n - a)$ that must be calculated for almost all $n$.  We detail the overall improvement in Section \ref{coding section}.  We introduce the potential internal improvement in Section \ref{internal section} and test it in Section \ref{empirical section}.  We end the paper by proposing a new approach for improving the current unconditional upper bound on $f(n)$.

\section{Algorithmic aspects}
\label{coding section}
\subsection{The de Reyna \& van de Lune algorithm.}
J. Arias de Reyna and J. van de Lune \cite{AV} developed code in Python to perform the analysis of their algorithm, which they have generously shared with us.  Additionally, Fuller has published open-source code \cite{F} written in C to calculate integer complexities.  Using these, we have developed code\footnote{See the \text{``calculate\_complexities.c''} file at https://github.com/kcordwel/Integer-Complexity} in C that is comparable to J. Arias de Reyna and J. van de Lune's Python code.  The heart of the code is the calc\_count method, which calculates $D(b, r)$ for varying values of $b$ and $r$, where $D(b,r)$ is an upper bound on how much multiplying by $b$ and adding $r$ increases the complexity of any given number. More precisely, we define $D(b, r)$ to be the smallest integer such that
\begin{equation}
f(r+bn) \leq f(n) + D(b,r)
\end{equation} for all $n$.  As an example, notice that $D(b, 0) \leq f(b)$, because we can always represent $b$ with $f(b)$ 1's and $n$ with $f(n)$ 1's and then multiply these two representations to achieve a representation of $bn$---and thus $f(bn) \leq f(n) + f(b)$.  Similarly, $D(1, r) \leq f(r)$ because we can represent $r$ with $f(r)$ 1's and $n$ with $f(n)$ 1's, and then add these two representations to achieve a representation of $n + r$ that uses $f(n) + f(r)$ 1's.

These integers $D(b, r)$ are useful for bounding $f(n)$ in the following way: \cite{AV} defined $C_{avg}$ as the infimum of all $C$ such that $f(n) \leq C\log(n)$ for a set of natural numbers of density 1 and showed that 
\begin{equation} 
C_{avg} \leq \frac{1}{b\log(b)}\sum_{r=0}^{b-1}D(b, r).
\end{equation}
In this calculation, we refer to $b$ as the base in which we are working.  
Our code closely follows the logic of J. Arias de Reyna and J. van de Lune's program, making the following slight optimization.

\begin{theorem} Take $b = 2^i3^j$ where $b < 10^{12}$ and $i + j > 0$.  If $r\ |\ b$ for $2 \leq r < b$, then $D(b, r) = f(b) + 1$.
\end{theorem}

\begin{proof} 
From the equality $r + bn = r(1 + n\cdot \frac{b}{r})$, notice that:
\begin{equation} f(bn + r) \leq 1 + f(r) + f\left(n\cdot \frac{b}{r}\right) \leq 1 + f(r) + f(n) + f(\frac{b}{r}) 
\end{equation}

From \cite{IBCOOP}, we know that $f(2^v3^w) = 2v + 3w$ for $2^v3^w < 10^{12}$ and $v + w > 0$.  We know that $b = 2^i3^j$, and since $r$ divides $b$, $r$ is of the form $2^x3^y$ for $x \leq i$, $y \leq j$.  This means that $b/r$ is of the form $2^{i - x}3^{j - y}$.  Since $r \geq 2,$ we have $x + y > 0$ and since $r < b$, we have $(i - x) + (j - y) > 0$.

Then applying the result of \cite{IBCOOP} to both $r$ and $\frac{b}{r}$, we obtain 
\begin{equation}
\begin{split}
 f(bn + r) &\leq 1 + f(r) + f(n) + f\left(\frac{b}{r}\right) \\
 &=  1 + 2x + 3y + f(n) + 2(i - x) + 3(i - y) \\
 &= 1 + 2i + 3j + f(n) \\
 &= 1 + f(b) + f(n).
\end{split}
\end{equation}
This shows that $D(b, r) \leq f(b) + 1$. 

Now we wish to argue that $D(b, r) \geq f(b) + 1$.  While this makes sense intuitively, in order to be rigorous we do a computer aided proof\footnote{See the code in the \text{``Thm2.1''} folder at https://github.com/kcordwel/Integer-Complexity}.  Our computer calculations work as follows:  First, we calculate the complexities of $b + r$ for all $b, r$ as in the theorem.  For the vast majority of the $b$, $r$, $f(b + r) = f(b) + 1$, meaning that $D(b, r) \geq f(b) + 1$ from the definition of $D(b, r)$.  However, there are 372 pairs of $b, r$ such that $f(b + r) \neq f(b) + 1$.  For these pairs, we do a second pass and calculate the complexities of $2b + r$.  For all of the pairs, $f(2b + r) = f(b) + 3 = f(b) + 1 + f(2)$, meaning that $D(b, r) \geq f(b) + 1$.   
\end{proof}

J. Arias de Reyna and J. van de Lune \cite{AV} suggest that their algorithms will be more powerful when implemented in C and Pascal.  \cite{AV} proved that their algorithm has running time $\mathcal{O}(N^{\alpha})$ where

\begin{equation}
\alpha = -1 + \frac{1}{\log b}\log\left(\sum_{d = 0}^{b-1} 3^{\frac{1}{3}D(b, d)}\right).
\end{equation}

They calculated the runtime of their algorithm for bases $2^n3^m$ up to 3188646, and found the best value of $\alpha$ as

$$\alpha = -1 + \frac{\log(48399164638047 +  3^{1/3}\cdot 33606823799088 + 3^{2/3}\cdot 23231513379231)}{\log(2^{10}\cdot 3^7)} \approx 1.230175$$
in base $2^{10}3^7 = 2239488$.  Using C is advantageous because it runs much faster than Python, and so we are able to calculate values for higher bases.  We calculated values for bases $2^n3^m \leq 57395628$.\footnote{After submission of this paper, we ran the code even longer, for bases $2^n3^m \leq 100663296$.  See the ``calculate\_complexities.txt'' file on GitHub for our data.}  In base $2^{13}3^8 = 53747712$, we find that the runtime is $\mathcal{O}(n^{\alpha})$ where 

$$\alpha = -1 + \frac{\log(50903564566217859 + 35271975106952037\cdot 3^{1/3} + 24493392174530898\cdot 3^{2/3})}{\log(2^{13}3^8)},$$
so that the runtime is $\mathcal{O}(n^{1.222911236}).$\footnote{Further, in base $80621568 = 2^{12}3^9$, we obtain $\mathcal{O}(n^{1.22188})$.  See the ``calculate\_complexities.txt'' file on GitHub for the exact numbers involved in this calculation.}

\subsection{Improved asymptotic results.} Probably Guy \cite{G} was the first who remarked that while pointwise bounds seem difficult, it is possible to establish bounds that are true for almost all (in the sense of asymptotic density 1) numbers. His method showed that $f(n) \leq 3.816 \log_{3}{n}$ for a subset of integers with density 1.

Using their definition of $C_{avg}$ as the infimum of all $C$ where $f(n) \leq C\log(n)$ for a set of natural numbers of density 1, \cite{AV} showed that for any base $b \geq 2$

\begin{equation}
C_{avg} \leq \frac{1}{b\log b}\sum_{r = 0}^{b-1} D(b, r).
\end{equation}

In base $b = 2^93^8,$ they obtain
\begin{equation} 
C_{avg} \leq \frac{166991500}{2^93^8\log(2^93^8)},
\end{equation}
so that $f(n) \leq 3.30808\log(n)$, or $f(n) \leq 3.63430\log_3(n)$,  for a set of natural numbers of density 1.

We find that in base $2^{11}3^9$,

\begin{equation} 
C_{avg} \leq \frac{2326006662}{2^{11}3^9\log(2^{11}3^9)},
\end{equation}
so that $f(n) \leq 3.29497 \log(n)$, or $f(n) \leq 3.61989\log_3(n)$, for a set of natural numbers of density 1.\footnote{Further, in base $80621568 = 2^{12}3^9$, we find $f(n) \leq 3.29180 \log(n)$, or $f(n) \leq 3.61642log_3(n)$.  See the ``calculate\_complexities.txt'' file for the exact numbers involved in this calculation.}

\section{Possible Improvements via Balancing Digits}
\label{internal section}
\subsection{Balancing Digits}
Our goal is to improve the algorithm for calculating complexity given in \cite{AV}. The rate-limiting factor in this algorithm is checking, for all $n \leq N$, $f(a) + f(n-a)$ for all $1 \leq a \leq \text{kMax}$ for some $\text{kMax}$ that is specially calculated for each $n$.  We will show that we can give a strong uniform bound on the number of summands that must be checked for almost all $n$.

We say that $n \in \mathbb{Z}$ is \textit{digit-balanced} in base $b$ if each of the digits $1, \dots, b - 1$ occurs roughly $1/b$ times in the base $b$ representation of $n$, or \textit{digit-unbalanced} if some digits occur significantly more often than others.  We will show that almost all numbers are digit-balanced, although the exact threshold of variation that we allow will depend on the base $b$.  Finally, assuming that we have a set $S$ of digit-balanced numbers in base $b$, we will use Guy's method to find that for any $n \in S$, $f(n) \leq c\log_3(n)$ for some $c$.  Then, using this bound on $f(n)$ and assuming that $f(n) = f(a) + f(n - a),$ we are able to bound $a$, which, in turn, narrows the search space that a reasonable algorithm has to cover.

\subsection{Bounds on Digit-Balanced Numbers}
Our main result is as follows.
\begin{proposition} There exists a constant $c_{b} > 0$ only depending on the base $b$ such that 
$$ \# \left\{ 1 \leq n \leq N: \max_{1 \leq i \leq b}{ \left| \frac{ \mbox{number of digits of}~n~\mbox{in base}~b~\mbox{that are i}}{\mbox{number of digits of}~n~\mbox{in base}~b} - \frac{1}{b} \right|} \geq \varepsilon \right\} \leq N^{1- c_b \varepsilon^2}.$$
\end{proposition}
\begin{proof} The main idea behind the argument is to replace a combinatorial counting argument by the probabilistic large deviation theory. Let $N\ =\ b^k$, and consider all $k$-digit numbers in base $b$, let $X_i$ be a random variable such that $X_i\ =\ 1$ with probability $1/b$ and 0 otherwise for $1 \leq i \leq k$.  For any given digit $0 \leq d < b$, each $X_i$ gives the probability that this digit will appear in a fixed position $i$ in the base $b$ representation of a number.  Since we are considering $k$-digit numbers, we need to understand the average value of $X_1 + \dots + X_k$ and to analyze how close this average is to $\frac{1}{b}$.  Let $\overline{X}\ =\ \frac{1}{k}(X_1 + \cdots + X_k)$.
Next, we can use Hoeffding's inequality, which gives
\begin{equation}
P\left(\overline{X} - \frac{1}{b}\ \geq\ \epsilon\right)\ \leq\ e^{-2k\epsilon^2}.
\end{equation}
We know that $k \approx \log_b(N)\ =\ \frac{\log(N)}{\log(b)}$, so:
\begin{equation}
e^{-2k\epsilon^2}\ =\ e^{-2\epsilon^2\frac{\log(N)}{\log(b)}}\ =\ (e^{\log(N)})^{-2\epsilon^2\frac{1}{\log(b)}}\ =\ N^{\frac{-2\epsilon^2}{\log(b)}}.
\end{equation}
So, the probability that a number with $k$ digits in its base $b$ representation has some digits that appear more often than the average is less than or equal to $N^{\frac{-2\epsilon^2}{\log(b)}},$ meaning that $|S|\ \leq\ N\cdot N^{\frac{-2\epsilon^2}{\log(b)}}\ =\ N^{1 - \frac{2\epsilon^2}{\log(b)}}.$
\end{proof}

\subsection{Bound on Number of Summands}\label{summand bound section}
Assume now that $f(n)\ =\ f(n-a) + f(a)$ and that this is the optimal representation using the least number of 1's. We assume that $f(n)\ =\ c\log_3(n)$ for some $c > 0$.  Our goal is to derive a bound on $a$.  The main idea is
to show that the logarithmic growth implies that $a$ cannot be very large (otherwise the growth of $f(n)$ would be closer to linear).  Using the lower bound due to Selfridge \cite{G}, we attain:

\begin{equation}
c\log_3(n)\ \geq\ 3(\log_3(n-a) + \log_3(a)).
\end{equation}
This is equivalent to:

\begin{equation}
\log_3(n^{c/3})\ \geq\ \log_3(n-a) + \log_3(a).
\end{equation}

Say that $a\ =\ qn$, where necessarily $q\ \leq\ \frac{1}{2}$.  Then we have:

\begin{equation}
\log_3(n^{c/3})\ \geq\ \log_3((1-q)n\cdot a).
\end{equation}
Exponentiating both sides and simplifying gives

\begin{equation}
\frac{n^{c/3-1}}{1-q}\ \geq\ a.
\end{equation}
Since $q\ \leq\ \frac{1}{2}$, then $1 - q\ \geq\ \frac{1}{2}$, and so

\begin{equation}
\frac{n^{c/3-1}}{1/2}\ \geq\ \frac{n^{c/3-1}}{1-q}\ \geq\ a,
\end{equation}
or:

\begin{equation}
2n^{c/3-1} \ \geq\ \frac{n^{c/3- 1}}{1-q}\ \geq\ a.
\end{equation}

Thus, we need only check for values of $a$ at most $2n^{c/3 - 1}$.
\subsection{Binary Analysis}
To see how our result works, we analyze it in the simplest possible base, which is binary.  Consider $k$-digit numbers less than $N$ (so that $k\ \approx\ \log_2(N)$).  The average case in Guy's method, illustrated in \cite{G} and based on Horner's scheme of representing binary numbers, gives $f(n) \leq 5\log_2(n)/2$, or $f(n) < 3.962407\log_3(n)$.  ``Bad'' numbers in base 2 are those that have many 1's, as that is when the representation is rather inefficient.  If we move away from the average case to numbers which have, say, $75\%$ 1's and $25\%$ 0's, then the constant in Guy's method is 

\begin{equation}
\frac{1}{\log(2)}(3\cdot .75 + 2\cdot .25)\log(3)\ <\ 4.358647.
\end{equation}

This is already much worse than the original average case constant of 3.962407, and so we need to stay much closer to the average case.  In particular, the following percentages of 1's and 0's give the following values for the constant in Guy's method:

\vspace{.75em}
\begin{center}
\begin{tabular}{ |c | c | c| }
\hline
Percent 0's & Percent 1's & Constant\\ \hline
46 & 54 & 4.02581 \\ \hline
47 & 53 & 4.00997 \\ \hline
48 & 52 & 3.99411\\ \hline
49 & 51 & 3.97826 \\ \hline
49.9 & 50.1 & 3.96399 \\ \hline
49.99 & 50.01 & 3.962565 \\  \hline
\end{tabular}
\end{center}
\vspace{.75em}

Consider numbers with at most 46\% 0's and 54\% 1's.  The previous section affords a bound of $a \leq 2n^{4.02581/3 - 1} \leq 2n^{0.342}$ for such numbers.  We want to understand how often this case occurs.  Recall that we are considering $k$-digit numbers.  We need to bound the number of times that 0 occurs at most $\frac{46k}{100}$ times, or the number of times that 1 occurs at least $\frac{54k}{100}$ times.  Say that $X_i$ is the Bernoulli variable corresponding to digit $i$, $1 \leq i \leq k$.  Then $P(X_i = 1) = \frac{1}{2}$.  Let $S_k = X_1 + \cdots + X_k$, so that $S_k$ represents the total number of 1's in our number.  Since $\frac{1}{2} < \frac{54}{100} < 1$, we may apply Theorem 1 from \cite{AG} to achieve the following bound:
\begin{equation}
P\left(S_k \geq \frac{54k}{100}\right) \leq e^{-kD\left(\frac{54}{100} \vert\vert \frac{1}{2}\right)}
\end{equation} 
where 

\begin{equation}
D\left(\frac{54}{100}\ \vert\vert\ \frac{1}{2}\right)\ =\ \frac{54}{100}\log\left(2\left(\frac{54}{100}\right)\right) + \left(1 - \frac{54}{100}\right)\log\left(2\left(1 - \frac{54}{100}\right)\right).
\end{equation}

Because $k\ \approx\ \log_2(N)$, we get that 

\begin{equation}
P\left(S_k \geq \frac{54k}{100}\right)\ \leq\ N^{-D(\frac{54}{100}\ ||\ \frac{1}{2})\frac{1}{\log 2}}.
\end{equation}
In particular, then, there are at most $N^{1 - \frac{1}{\log(2)}\cdot D(\frac{54}{100} \vert\vert \frac{1}{2})} < N^{1 - .004622}$ ``bad'' numbers, i.e. we have the desired bound $a \leq 2n^{0.342}$ for the other $> N^{.004622}$ numbers, which is significant as $N$ grows large.  Call this set of numbers for which we have this bound $\mathcal{U}$.

Following the analysis in \cite{AV}, Arias de Reyna and van de Lune's algorithm has a runtime of $n^{\alpha}$ in base 2 where 
\begin{equation}
\alpha = -1 + \frac{\log(3^{D(2, 0)/3} + 3^{D(2, 1)/3})}{\log(2)} = -1 + \frac{\log(3^{2/3} + 3)}{\log(2)} \approx 1.3448.
\end{equation}

Recall that in their complexity proof, Arias de Reyna and van de Lune denote the number of summands that must be checked for each $n$ by kMax.  Our bound on the numbers in $\mathcal{U}$ compares well to \cite{AV}'s bound in that if kMax were uniform for all numbers in \cite{AV}, our bound would be lower on all $u \in \mathcal{U}$.  More explicitly, in binary, if kMax were uniform, then \cite{AV} would require checking summands up to $\approx n^{0.3448}$ whereas we require checking summands up to $\approx 2n^{0.342}$ for numbers in $\mathcal{U}$.

Unfortunately, kMax is not uniform in this way, and so we cannot claim a definitive improvement with our uniform bound on $\alpha$.  It is possible that some of the $u \in \mathcal{U}$ have a low value of kMax to begin with, and for such numbers our bound may not afford an improvement.  Conversely, it is possible that our bound will improve some numbers that are not in $\mathcal{U}$.  Overall, since kMax is not uniform, it is not easy to theoretically compare our bound to \cite{AV}.  Given this, and given that the ideal bases are much larger than binary (which significantly complicates theoretical analysis), we performed a number of empirical tests to understand how our algorithm compares to \cite{AV} in the general case.


\section{Empirical Calculations}\label{empirical section}
To see whether our method improves J. Arias de Reyna and J. van de Lune's algorithm in practice, we modified J. Arias de Reyna and J. van de Lune's code by adding various precomputations and calculating how many numbers would be improved with these precomputations\footnote{See the ``ExperimentalResults'' folder at https://github.com/kcordwel/Integer-Complexity}.

The first precomputation uses a greedy algorithm due to Steinerberger \cite{St}, which gives that $f(n) \leq 3.66\log_3(n)$ for most $n$.  The recursive algorithm works as follows: if $n \equiv 0 \mod 6$ or $n \equiv 3 \mod 6$, take $n = 3(n/3)$ and run the algorithm on $n/3$.  If $n \equiv 2 \mod 6$ or $n \equiv 4 \mod 6$, take $n = 2(n/2)$ and run the algorithm on $n/2$.  If $n \equiv 1 \mod 6$, take $n = 1 + 3(n-1)/3$ and run the algorithm on $(n-1)/3$.  If $n \equiv 5 \mod 6$, take $n = 1 + 2(n-1)/2$ and run the algorithm on $(n-1)/2$.

The method is as follows: First, run the greedy algorithm on all of the numbers up to some limit and store the results in a dictionary.  Then, use these values to compute a bound on the number of summands for each number (using the formula derived in Section \ref{summand bound section}).  Store a counter that is initialized to 0.  Next, run J. Arias de Reyna and J. van de Lune's algorithm.  For each number, test whether the precomputed summand bound is better than the summand bound in the original algorithm.  If an improvement is found, increment the counter.  When we use this algorithm to precompute summands, we improve 7153 numbers out of the first 200000, or less than 3.6\% of numbers.  If we compute complexities further, up to 2000000, we improve 60864 numbers, or less than 3.05\% of numbers.


We can also combine Steinerberger's algorithm with a stronger algorithm, due to Shriver \cite{Sh}.  Shriver developed a greedy algorithm in base 2310.  If we use the best upper bound on complexities from Shriver and Steinerberger's greedy algorithm, we improve 11188 numbers out of 200000, or about 5.6\% of numbers.  If we compute complexities up to 2000000, we improve 107077 numbers, or less than 5.36\% of numbers.

Shriver conjectures that his best algorithm, which uses simulated annealing, produces a bound of $f(n) \leq 3.529\log_3(n)$ for generic integers.  In fact, only 824 numbers up to 2000000 would be improved by assuming a uniform bound of $f(n) \leq 3.529\log_3(n)$.  Of course, this is a purely theoretical result---if we were to actually introduce a uniform bound, then we would not be able to accurately calculate complexities.  If we become even more optimistic and use a uniform bound of $f(n) \leq 3.5\log_3(n)$, we would only potentially improve 4978 numbers out of the first 2000000.  Similarly, using $f(n) \leq 3.4\log_3(n)$ would improve 124707 numbers of 2000000, which is about 6.23\%.  If we venture significantly below Shriver's conjecture of $3.529\log_3(n)$ and use $f(n) \leq 3.3\log_3(n)$ uniformly, then we start to see a significant difference---we would improve 726756 numbers of 2000000, or about 36\%.

Overall, it seems that Arias de Reyna and van de Lune's algorithm already has a strong bound on the number of summands that are computed.  It is possible that we are encountering difficulties because kMax is not uniform, or it is possible that the complexity of J. Arias de Reyna and J. van de Lune's algorithm is significantly lower than $\mathcal{O}(n^{1.223})$. Thus, while summand precomputing improves the complexity computation for some numbers, given the overhead for performing precomputations and the current speed of J. Arias de Reyna and J. van de Lune's algorithm, introducing a precomputation does not seem to yield an overall improvement to the algorithm.
\section{Progress Towards an Unconditional Upper Bound}
The current unconditional upper bound on complexity, $f(n) \leq 3\log_2(x)$, is derived from applying Guy's method in base 2 to $n$.  In particular, the most complex numbers have binary expansions of the form $11\cdots1_2$ so that at each step, Guy's method requires three 1's.  The resulting representation is of the form $1 + (1 + 1)[1 + (1 + 1)[ \cdots]]$.

Say that $n \text{ mod } 3 \equiv k$.  Instead of applying Guy's method to $n$, what if write $n = k + (1 + 1 + 1)(n - k)/3$ and then apply Guy's method to $(n - k)/3$?  Then in the case where $n = 11\cdots 1_2$, $(n-k)/3$ is either of the form $1010\cdots1$ or $1010\cdots 0$, and applying Guy's method to $(n-k)/3$ gives $f((n-k)/3) \leq 1 + 2.5\log_2(n)$.  Using this, we find that $f(n) \leq 6 + 2.5\log_2(n)$, which is a significant improvement over $f(n) \leq 3\log_2(n)$.

This suggests the following method: If the binary representation of $n$ contains more than a certain percentage of 1's, then write $n$ as $k + (1 + 1 + 1)\cdot (n - k)/3$ and apply Guy's method instead to $(n - k)/3$.  Empirically, in most cases, when the binary expansion of $n$ contains a high percentage of 1's, $(n - k)/3$ has a significantly lower percentage of 1's.  However, there are some examples where this fails.  For example, if $n = 2^{102}-2^{100}-2$, then both the binary expansion of $n$ and the binary expansion of $(n-1)/3$ have a high percentage of 1's.  Notably, if we repeat this division process and consider $((n-1)/3)/3$, then we will obtain a number with a nice binary expansion.  Accordingly, we say that $2^{102}-2^{100}-2$ requires two iterations of division by 3.

Some numbers require numerous iterations of division by 3 before their binary expansions are nice.  For example, $n = 2^{3000} - 2^{2975} - 2^{2807} - 1$ requires nine iterations.  These sorts of counterexamples seem to follow some interesting patterns.   Let $n_i$ denote the number obtained after $i$ iterations of division by 3 so that $n_0 = n$, $n_1 = (n_0 - (n_0 \text{ mod } 3))/3$, etc.  In general, it seems that the number of iterations that are necessary to produce a ``nice'' binary expansion is tied to the number of iterations for which $n \equiv 2 \text{ mod } 3$.  For example, when $n = 2^{3000} - 2^{2975} - 2^{2807} - 1$, then $n_0 \equiv n_1 \equiv n_2 \equiv \cdots \equiv n_7 \equiv 2 \text{ mod } 3$, but $n_8 \equiv 0 \text{ mod } 3$, and $n_9$ has the first ``nice'' binary expansion.

It should be noted that there is no reason to only employ division by 3.  For example, when $n = 2^{3000} - 2^{2975} - 2^{2807} - 1$, $n \text{ mod } 11 \equiv 5$, and $(n - 6)/11$ has a nice binary expansion.  It should be noted that $n \equiv 4 \text{ mod } 5$ and $n \equiv 6 \text{ mod } 7$, and the binary representations of $(n - 4)/5$ and $(n - 6)/7$ both contain a large percentage of 1's.

In general, then, performing this process of division by appropriate numbers before applying Guy's method is a promising strategy for obtaining an improvement on the unconditional upper bound on $f(n)$.
We believe that it could be an interesting problem to make these vague heuristics precise and understand whether this could give rise to a new effective method of giving explicit constructions
of $n$ with sums and products that use few $1$'s.

\section{Acknowledgments}
We would like to thank Professor Arias de Reyna for generously sharing the code that he developed with Professor van de Lune.  Thank you to the SMALL REU program, Williams College, and the Williams College Science Center where the bulk of this work took place. We would like to thank Professor Amanda Folsom for funding from NSF Grant DMS1449679 as well as SMALL REU for funding from NSF Grant DMS1347804, the Williams College Finnerty Fund, and the Clare Boothe Luce Program. The fourth listed author was supported by NSF grants DMS1265673 and DMS1561945, the fifth listed author was supported by Simons Foundation Grant \#360560 and the seventh listed author was supported by NSF Grant DMS1763179 and the Alfred P. Sloan Foundation. Finally, we thank an anonymous referee for suggestions that significantly improved the paper.
\newpage

\ \\
\end{document}